 \documentclass[12pt,a4paper]{amsart}
 \usepackage{amsfonts,amsmath}
\numberwithin{equation}{section}

\theoremstyle{plain}
\newtheorem{Th}{Theorem}[section]
\newtheorem{Lemma}[Th]{Lemma}

 \theoremstyle{definition}
\newtheorem{Def}[Th]{Definition}

\newtheorem{?}[Th]{Problem}

  \newcommand{\ab}[1]{\left\vert{#1}\right\vert}

   \newcommand{\co}{\operatorname{conv}}
   \newcommand{\ve}{\operatorname{vert}}

\begin{document}

\title{Sumsets and the convex hull}

\author{M\'at\'e Matolcsi}
\address{Alfr\'ed R\'enyi Institute of Mathematics\\
     Budapest, Pf. 127\\
     H-1364 Hungary\\
     (also at BME Department of Analysis, Budapest, H-1111, Egry J. u. 1)
} \email{matomate@renyi.hu}
\thanks{Supported by Hungarian National Foundation for Scientific Research
(OTKA), Grants No. PF-64061, T-049301, T-047276}

\author{Imre Z. Ruzsa}
\address{Alfr\'ed R\'enyi Institute of Mathematics\\
     Budapest, Pf. 127\\
     H-1364 Hungary
}
\email{ruzsa@renyi.hu}

\thanks{Supported by Hungarian National Foundation for Scientific Research
(OTKA), Grants No. K 61908, K 72731.}

    \begin{abstract}
     We extend Freiman's inequality on the cardinality of the sumset of a $d$
dimensional set. We consider different sets related by an inclusion of their
convex hull, and one of them added possibly several times.
    \end{abstract}

     \maketitle

     \section{Introduction}

     The aim of this paper is to give a lower estimate for the cardinality of
certain sumsets in ${\mathbb {R}}^d$.

     We say that a set in ${\mathbb {R}}^d$ is \emph{proper $d$-dimensional}
     if it is not contained in any affine hyperplane.

     Our starting point is the following classical theorem of Freiman.

     \begin{Th}[Freiman\cite{freiman73}, Lemma 1.14] \label{freimand}
     Let $A\subset  {\mathbb {R}}^d$ be a finite set, $|A|=m$. Assume that $A$ is proper
$d$-dimensional. Then
     $$   |A+A| \geq  m(d+1) - {d(d+1)\over 2} .  $$
     \end{Th}

     We will show that to get this inequality it is sufficient to use the
vertices (extremal points) of $A$.

     \begin{Def} 
     We say that a point $a\in A$ is a \emph{vertex} of a set $A\subset {\mathbb {R}}^d$ if it is not
in the convex hull of $A\setminus \{a\}$. The set of vertices will be denoted by $\ve
A$.
     \end{Def}

     The convex hull of a set $A$ will be denoted by $\co A$.

     \begin{Th} \label{csucs}
     Let $A\subset  {\mathbb {R}}^d$ be a finite set, $|A|=m$. Assume that $A$ is proper
$d$-dimensional, and let $A' = \ve A$,
We have
     $$   |A+A'| \geq  m(d+1) - {d(d+1)\over 2} .  $$
    \end{Th}

     This can be extended to different summands as follows.

     \begin{Th} \label{ketto}
     Let $A, B\subset  {\mathbb {R}}^d$ be  finite sets, $|A|=m$. Assume that $B$ is
proper $d$-dimensional and  $A \subset  \co B$. We have
     $$   |A+B| \geq  m(d+1) - {d(d+1)\over 2} .  $$
     \end{Th}

     Finally we extend it to several summands as follows. We use $kB=B+\dots +B$
to denote repeated addition. As far as we know even the case of $A=B$ seems to
be new here.

     \begin{Th} \label{tobb}
     Let $A, B\subset  {\mathbb {R}}^d$ be  finite sets, $|A|=m$. Assume that $B$ is
proper $d$-dimensional and  $A \subset  \co B$. Let $k$ be a positive integer.
We have
     \begin{equation} \label{fo}
     |A+kB| \geq  m \binom{d+k}{k}  - k \binom{d+k}{k+1}
     = \left(  m - {kd \over  k+1}\right)  \binom{d+k}{k}.  \end{equation}
     \end{Th}

     The case $d=1$ of the above theorems is quite obvious. In \cite{m3} we
gave a less obvious result which compares a complete sum and its
subsums, which sounds as follows.

     \begin{Th}  \label{superadd}
     Let $A_1, \dots , A_k$ be finite, nonempty sets of integers. Let $A_i'$ be the set consisting of the smallest and the
     largest elements of $A_i$ (so that $1\leq |A_i'|\le2$). Put
     \[   S = A_1 + \dots  + A_k ,  \]
     \[   S_i = A_1 + \dots  + A_{i-1} + A_{i+1}+ \dots  + A_k ,  \]
     \[   S_i' = A_1 + \dots  + A_{i-1} + A_i' + A_{i+1}+ \dots  + A_k , \]
     \[   S' = \bigcup _{i=1}^k S_i' .  \]
     We have
     \begin{equation} \label{also}
     \left| S \right| \geq  \left|S' \right| \geq  {1\over k-1} \sum _{i=1}^k \left|S_i \right|
- {1\over k-1}.
     \end{equation}
     \end{Th}

     \begin{?} 
     Generalize Theorem \ref{superadd} to multidimensional sets. A proper
generalization should give the correct order of magnitude, hence the analog of
 \eqref  {also} could be of the form
     \[ \left| S \right| \geq  \left|S' \right| \geq  \left( {k^{d-1}\over (k-1)^d} - \varepsilon
\right) \sum _{i=1}^k \left|S_i \right|       \]
     if all sets are sufficiently large.
     \end{?}

     \begin{?} 
     Let $A, B_1, \dots , B_k \subset {\mathbb {R}}^d$ such that the $B_i$ are proper $d$-dimesional
and
     \[   A \subset  \co B_1 \subset  \co B_2 \subset  \dots  \subset  \co B_k .  \]
     Does the esimate given in \eqref  {fo} also hold for $A+B_1+\dots +B_k$?
     \end{?}

     This is easy for $d=1$.

     \section{A simplicial decomposition}

     We will need a result about simplicial decomposition.

     By a \emph{simplex} in ${\mathbb {R}}^d$ we mean a proper $d$-dimensional compact set
which is the convex hull of $d+1$ points.

     \begin{Def} 
     Let $S_1, S_2 \subset  {\mathbb {R}}^d$ be simplices, $B_i = \ve S_i$. We say that they are
in \emph{regular position}, if
     \[   S_1 \cap  S_2 = \co (B_1 \cap  B_2) , \]
     that is, they meet in a common $k$-dimensional face for some $k\leq d$. (This
does not exclude the extremal cases when they are disjoint or they
coincide.) We say that a collection of simplices is in regular
position if any two of them are.
     \end{Def}

     \begin{Lemma}  \label{felbontas}
     Let $B\subset {\mathbb {R}}^d$ be a proper $d$ dimensional finite set, $S=\co B$. There is a
sequence
$S_1, S_2, \dots , S_n$ of distinct simplices in regular position with the
following properties.

     a) $ S = \bigcup  S_i$.

     b) $B_i = \ve S_i = S_i \cap  B$.

     c) Each $S_i$, $2\leq i\leq n$ meets at least one of $S_1$, \dots , $S_{i-1}$ in a
$(d-1)$ dimensional face.
     \end{Lemma}

     We mentioned this lemma to several geometers and all answered ``of
course'' and offered a proof immediately, but none could name a reference with
this formulation, so we include a proof for completeness. This proof was
communicated to us by prof. K\'aroly B\"or\"oczki.

     \begin{proof} 
     We use induction on $ \ab B$. The case $ \ab B =2$ is clear. Let $|B|=k$, and assume we
know it for smaller sets (in any possible dimension).

     Let $b$ be a vertex of $B$ and apply it for the set $B'=B\setminus  \{b\}$. This
set may be $d$ or $d-1$ dimensional.

     First case: $B'$ is $d$ dimensional. With the natural notation let
     \[   S' = \bigcup _{i=1}^{n'} S_i' \]
     be the prescribed decomposition of $S' = \co B'$. We start the
decomposition of $S$ with these, and add some more as follows.

     We say that a point $x$ of $S'$ is \emph{visible} from $b$, if $x$ is the
only point of the segment joining $x$ and $b$ in $S'$. Some of the simplices
$S_i'$ have (one or more) $d-1$ dimensional faces that are completely visible
from $b$. Now if $F$ is such a face, then we add the simplex
     \[   \co (F \cup  \{b\})  \]
     to our list.

     Second case: $B'$ is $d-1$ dimensional. Again we start with the
decomposition of $S'$, just in this case the sets $S_i'$ will be $d-1$
dimensional simplices. Now the decomposition of $S$ will simply consist of
     \[   S_i = \co (S_i' \cup  \{b\}), \ n = n' .  \]
     \end{proof}

     The construction above immediately gave property c). We note that it is
not really an extra requirement, every decomposition has it after
a suitable rearrangement. This just means that the graph obtained
by using our simplices as vertices and connecting two of them if
they share a $d-1$ dimensional face is connected. Now take two
simplices, say $S_i$ and $S_j$. Take an inner point in each and
connect them by a segment. For a generic choice of these point
this segment will not meet any of the $\leq d-2$ dimensional faces
of any $S_k$. Now as we walk along this segment and go from one
simplex into another, this gives a path in our graph between the
vertices corresponding to $S_i$ and $S_j$.

     \section{The case of a simplex}

     Here we prove Theorem \ref{tobb} for the case $\ab B = d+1$.

     \begin{Lemma}  \label{foeset}
     Let $A, B\subset  {\mathbb {R}}^d$ be  finite sets, $|A|=m$, $\ab B =d+1$. Assume
that $B$ is
proper $d$-dimensional and  $A \subset  \co B$. Let $k$ be a positive integer. Write
$\ab {A\cap B} = m_1$.
 We have
     \begin{equation} \label{s1}
     |A+kB| = (m-m_1) \binom{d+k}{k}  + \binom{d+k+1}{k+1} -
\binom{d-m_1+k+1}{k+1} .  \end{equation}
     In particular, if $\ab{A\cap B} \leq 1$, then
     \begin{equation} \label{s2}
     |A+kB| = m \binom{d+k}{k} .   \end{equation}
     We have always
     \begin{equation} \label{s3}
     |A+kB| \geq  m \binom{d+k}{k}  - k \binom{d+k}{k+1} = \left(  m - {kd
\over  k+1}\right)  \binom{d+k}{k}.  \end{equation}
     \end{Lemma}

     \begin{proof}
     Put $A_1=A\cap B$, $A_2 = A \setminus  B$.
     Write $B=\{b_0, \dots , b_d\}$, arranged in such a way that
     \[   A_1 = A \cap  B = \{b_0, \dots , b_{m_1-1}\} . \]
     The elements of $kB$ are the points of the form
\[   s = \sum  _{i=0}^d x_i b_i, \ x_i \in {\mathbb {Z}}\ , x_i \geq 0, \ \sum  x_i = k , \]
     and this representation is unique. Clearly
     \[   \ab {kB} = \binom{d+k}{k} .  \]

     Each element of $A$ has a unique representation of the form
     \[   a = \sum  _{i=0}^k \alpha _i d_i, \ \alpha _i \in {\mathbb {R}}\ , \alpha _i \geq 0, \ \sum  \alpha _i = 1 , \]
\[   a = \sum  _{i=0}^d \alpha _i b_i, \ \alpha _i \in  {\mathbb {R}}\ , \alpha _i \geq 0, \ \sum  \alpha _i = 1 , \]
     and if $a\in A_1$, then some $\alpha _i=1$ and the others are equal to 0, while if
$a\in A_2$, then at least two $\alpha _i$'s are positive.

     Assume now that $a+s=a'+s'$ with certain $a,a'\in A$, $s,s'\in kB$. By
substituting the above representations we obtain
     \[   \sum  (\alpha _i+x_i) b_i = \sum  (\alpha '_i+x'_i) b_i, \ \sum  (\alpha _i+x_i) = \sum  (\alpha '_i+x'_i) =
k+1 , \]
     hence $\alpha _i+x_i =\alpha '_i+x'_i $ for all $i$. By looking at the integral and
fractional parts we see that this is possible only if $\alpha _i=\alpha '_i$, or one of
them is 1 and the other is 0. If the second possibility never happens, then
$a=a'$. If it happens, say $\alpha _i=1, \alpha '_i=0$ for some $i$, then $\alpha _j=0$ for all
$j\ne i$ and then each $a_j'$ must also be 0 or 1, that is, $a,a'\in A_1$.

     The previous discussion shows that $(A_1+kB) \cap  (A_2+kB)= \emptyset $ and the sets
$a+kB$, $a\in A_2$ are disjoint, hence
     \[   \ab{A+kB}= \ab{A_1+kB} + \ab{A_2+kB} \]
     and
     \begin{equation} \label{s4}
     \ab{A_2+kB} = \ab{A_2} \ab{kB} = (m-m_1) \binom{d+k}{k} . \end{equation}

     Now we calculate $\ab{A_1+kB}$. The elements of this set are of the form
     \[    \sum  _{i=0}^d x_i b_i, \ x_i \in {\mathbb {Z}}\ , x_i \geq 0, \ \sum  x_i = k+1 , \]
     with the additional requirement that there is at least one subscript $i$,
$i\leq m_1-1$ with $x_i\geq 1$. Without this requirement the number would be the same
as
     \[   \ab {{(k+1)}B} = \binom{d+{k+1}}{{k+1}} .  \]
     The vectors
     $(x_0, \dots , x_d)$ that violate this requirement are those
that use only the last $d-m_1$ coordinates, hence their number is
     \[    \binom{d-m_1+{k+1}}{{k+1}} .  \]
     We obtain that
     \[ \ab{A_1+kB} = \binom{d+{k+1}}{{k+1}} - \binom{d-m_1+{k+1}}{{k+1}} .\]
     Adding this formula to \eqref  {s4} we get \eqref  {s1}.

     If $m_1=0$ or 1, this formula reduces to the one given in \eqref  {s2}.

     To show inequality \eqref  {s3}, observe that this formula is a decreasing
function of $m_1$, hence the minimal value is at $m_1=d+1$, which after an
elementary transformation corresponds
to the right side of \eqref  {s3}. Naturally this is attained only if $m\geq d+1$, and
for small values of $m$ the right side of \eqref  {s3} may even be negative.
     \end{proof}

     \section{The general case}

     \begin{proof}[Proof of Theorem \ref{tobb}]

     We apply Lemma \ref{felbontas} to our set $B$. This decomposition induces
a decomposition of $A$ as follows. We put
     \[   A_1 = A \cap  S_1, A_2 = A \cap  (S_2\setminus S_1), \dots , A_n = A \cap  \bigl( S_n \setminus
(S_1 \cup  S_2 \cup  \dots  \cup  S_{n-1}) \bigr) . \]
     Clearly the sets $A_i$ are disjoint and their union is $A$.
     Recall the notation $B_i=\ve S_i$.

     We claim that the sets $A_i+kB_i$ are also disjoint.

     Indeed, suppose that $a+s=a'+s'$ with $a\in A_i$, $a'\in A_j$, $s\in kB_i$,
$s'\in kB_j$, $i<j$. We have
     \[   {a+s \over  k+1} \in  S_i , \ {a'+s' \over  k+1} \in  S_j , \]
     and these points are equal, so they are in
     \[   S_i \cap  S_j = \co (B_i \cap  B_j) .  \]
     This means that in the unique convex representation of $ (a'+s')/(k+1)$ by
points of $B_j$ only elements of $B_i\cap  B_j$ are used. However, we can obtain
this representation via using the representation of $a'$ and the components of
$s'$, hence we must have $a'\in  \co (B_i\cap B_k) \subset  S_i$, a contradiction.

     This disjointness yields
     \[   \ab {A+kB}  \geq  \sum  \ab {A_i+kB_i} . \]

     We estimate the summands using Lemma \ref{foeset}.

     If $i>1$, then $\ab {A_i \cap  B_i} \leq 1$. Indeed, there is a $j<i$ such that
$S_j$ has a common $d-1$ dimensional face with $S_i$, and then the
$d$ vertices of this face are excluded from $A_i$ by definition.
So in this case \eqref  {s2} gives
     \[ |A_i+kB_i| = \ab {A_i} \binom{d+k}{k} .    \]
     For $i=1$ we can only use the weaker estimate \eqref  {s3}:
     \[ |A_1+kB_1| \geq  \ab{A_1} \binom{d+k}{k}  - k \binom{d+k}{k+1} .   \]
     Summing these equations we obtain \eqref  {fo}.
     \end{proof}

     {\bf Acknowledgement.} The authors profited much from discussions with
Katalin Gyarmati and K\'aroly B\"or\"oczky. In particular, we are indebted to Prof.
B\"or\"oczky for the present version of the proof of Lemma \ref{felbontas}.



 %
 %

     \end{document}